\documentclass{amsart}

\usepackage{amsfonts,amsmath,amssymb,amsthm,bbm}
\usepackage{mathrsfs,mathtools,stmaryrd}
\usepackage{xspace,mydef}
\usepackage{xcolor}
\usepackage{hyperref,cleveref}
\usepackage{etoolbox}
\DeclareMathAlphabet{\mathpzc}{OT1}{pzc}{m}{it}

\newcommand{\TheTitle}{Optimal control of fractional diffusion with Dirac measures}
\newcommand{\ShortTitle}{Optimal control of fractional diffusion with Dirac measures}

\theoremstyle{definition}
\newtheorem{definition}{Definition}
\newtheorem{remark}{Remark}

\theoremstyle{plain}
\newtheorem{theorem}{Theorem}
\newtheorem{proposition}{Proposition}

\begin{document}

\title[\ShortTitle]{\TheTitle}

\author{Enrique Ot\'arola}
\address[E.~Ot\'arola]{
  Departamento de Matem\'atica, Universidad T\'ecnica Federico Santa Mar\'ia,
  Valpara\'iso,
  Chile.
}
\email{enrique.otarola@usm.cl}
\urladdr{http://eotarola.mat.utfsm.cl/}

\author{Abner J. Salgado}
\address[A.J~Salgado]{
  Department of Mathematics, University of Tennessee,
  Knoxville, TN,
  37996,
  USA.
}
\email{asalgad1@utk.edu}
\urladdr{https://math.utk.edu/people/abner-salgado/}



\begin{abstract}
We study a PDE-constrained optimization problem for an elliptic equation with the spectral fractional Laplacian and a linear combination of Dirac measures as the forcing term; the controls are the amplitudes of these singular sources. We prove existence and uniqueness of an optimal solution and derive first-order optimality conditions. We then propose a discretization based on finite elements. Since the set of admissible controls is finite dimensional, the control variable itself does not require discretization. We conclude by deriving a priori error bounds.
\end{abstract}

\maketitle

\begin{center}
  \emph{We dedicate this work to the memory of Ricardo Dur\'an.}
\end{center}

\section{Introduction}
\label{sec:introduction}

Let $\Omega \subset \mathbb{R}^{d}$, with $d \in \{2,3\}$, be an open, bounded, and convex polytope with boundary $\partial \Omega$, and let $s \in (\tfrac{d}{4},1)$.  We denote by $(-\Delta)^{s}$ the spectral fractional Laplacian, with homogeneous Dirichlet boundary conditions built into this definition.  Let $\calD$ be a finite subset of $\Omega$ with cardinality $\ell \coloneqq \# \calD$, and let $\delta_{z}$ denote the Dirac measure supported at $z \in \calD$. Given a desired state $u_{\Omega} \in L^{2}(\Omega)$ and a regularization parameter $\alpha > 0$, we introduce the cost functional
\begin{equation}\label{eq:cost_functional}
  J(u,\mathbf{q}) \coloneqq \frac{1}{2} \| u - u_{\Omega} \|_{L^{2}(\Omega)}^{2}
          + \frac{\alpha}{2} \| \mathbf{q} \|_{\mathbb{R}^{\ell}}^{2}.
\end{equation}
The PDE-constrained optimization problem we consider is as follows: Minimize $J(u,\mathbf{q})$ subject to the fractional PDE
\begin{equation}\label{eq:state_equation}
  (-\Delta)^{s} u = \sum_{z \in \calD} q_{z} \delta_{z}
  \quad \text{in } \Omega,
\end{equation}
and the control constraints $\mathbf{q} \in \mathbf{Q}_{ad}$, where
$
  \mathbf{Q}_{ad} \coloneqq \left\{ \mathbf{q} \in \mathbb{R}^{\ell} :
  a_{z} \leq q_{z} \leq b_{z} \ \ \forall z \in \calD \right\}.
$
The amplitudes of the singular sources in \eqref{eq:state_equation} are collected in the vector $\mathbf{q} = \{ q_{z} \}_{z \in \calD}$, which is our control variable. The control bounds $\mathbf{a} = \{ a_{z} \}_{z \in \calD}$ and $\mathbf{b} = \{ b_{z} \}_{z \in \calD}$ are vectors in $\mathbb{R}^{\ell}$ for which we assume that $a_{z} < b_{z}$ for every $z \in \calD$.

Control problems with point sources arise in applications such as active sound control, where the source amplitudes act as controls \cite{MR2086168}. In our setting, the nonlocal operator $(-\Delta)^s$ and the singular forcing require a formulation beyond the standard energy framework; here we extend the approach of \cite{OtarolaSalgado2026} to $d=3$. This approach requires $s > \tfrac{d}{4}$, which also ensures that the adjoint state is continuous, so that the point evaluations in the optimality conditions are well defined.

\section{Notation and preliminaries}
\label{sec:notation}

We write $A \lesssim B$ whenever $A \leq c B$ with a constant $c>0$ independent of $A$, $B$, and any discretization parameter; $A \gtrsim B$ means $B \lesssim A$, and $A \eqsim B$ stands for $A \lesssim B \lesssim A$.  The Euclidean norm in $\mathbb{R}^{\ell}$ is denoted by $\| \cdot \|_{\mathbb{R}^{\ell}}$. Finally, $\calM(\Omega)$ denotes the space of finite Radon measures on $\Omega$, which can be identified with the dual of $C_{0}(\Omega)$, the space of continuous functions on $\bar{\Omega}$ vanishing on $\partial \Omega$. The symbol $\langle \cdot, \cdot \rangle$ denotes the associated duality pairing.

\subsection{The spectral fractional Laplacian and fractional Sobolev spaces}
\label{sub:spectral}

Consider the problem: Find $(\lambda,\varphi) \in \mathbb{R} \times (H_{0}^{1}(\Omega) \setminus \{0\})$ such that
$
  (\nabla \varphi , \nabla v)_{L^{2}(\Omega)}
  = \lambda \, (\varphi , v)_{L^{2}(\Omega)}
$
for all $v \in H_{0}^{1}(\Omega)$. This problem has a countable family of eigenpairs $\{ (\lambda_{k},\varphi_{k}) \}_{k \in \mathbb{N}} \subset \mathbb{R}_{+} \times (H_{0}^{1}(\Omega)\setminus\{0\})$ such that $\lambda_k \to \infty$ as $k \to \infty$ and $\{\varphi_{k}\}_{k \in \mathbb{N}}$ is an orthonormal basis of $L^{2}(\Omega)$ and an orthogonal basis of $H_{0}^{1}(\Omega)$. For $r \geq 0$, we define the spaces
\begin{equation}\label{eq:Hr_spaces}
  \begin{aligned}
  \mathbb{H}^{r}(\Omega) &\coloneqq
  \left\{ w \in L^2(\Omega): w = \sum_{k=1}^{\infty} w_{k} \varphi_{k},
  \
  \sum_{k=1}^{\infty} \lambda_{k}^{r} |w_{k}|^{2} < \infty \right\} ,
  \\
  \| w \|_{\mathbb{H}^{r}(\Omega)} &\coloneqq
  \left( \sum_{k=1}^{\infty} \lambda_{k}^{r} |w_{k}|^{2} \right)^{\frac{1}{2}},
  \end{aligned}
\end{equation}
where, for $k \in \mathbb{N}$, $w_{k} \coloneqq \int_{\Omega} w \varphi_{k} \, \mathrm{d}x$. For $r > 0$, $\mathbb{H}^{-r}(\Omega)$ denotes the dual space of
$\mathbb{H}^{r}(\Omega)$. Given $s \in (0,1)$ and $w \in C_{0}^{\infty}
(\Omega)$, we define the spectral fractional Laplacian \cite{MR2646117,MR2754080,MR2825595} by
\begin{equation}\label{eq:spectral_fl}
  (-\Delta)^{s} w \coloneqq \sum_{k=1}^{\infty} \lambda_{k}^{s} w_{k} \varphi_{k}.
\end{equation}
By density, $(-\Delta)^{s}$ extends to an isomorphism from $\mathbb{H}^{s}(\Omega)$ onto $\mathbb{H}^{-s}(\Omega)$.

The spaces $\mathbb{H}^{r}(\Omega)$ are the natural ones in this setting. However, in the analysis that follows we need to identify these with the classical fractional Sobolev spaces \cite{MR1742312,MR2328004,MR3343061,MR3356020}.
With equivalent norms, we have
\begin{equation}
  \mathbb{H}^{r}(\Omega) =
  \begin{dcases}
    H^{r}(\Omega), & r \in (0,\tfrac12),
    \\
    H^{1/2}_{00}(\Omega), & r=\frac12,
    \\
    H^{r}_{0}(\Omega), & r \in (\tfrac12,1),
    \\
    H_{0}^{1}(\Omega) \cap H^{r}(\Omega), & r \in (1,2).
  \end{dcases}
\label{eq:characterization}
\end{equation}
This last identification holds because $\Omega$ is convex.

\section{Fractional diffusion under a measure-valued right-hand side}
\label{sec:fractional_diffusion_under_singular_forcing}

We consider the following fractional PDE with a measure-valued right-hand side:
\begin{equation}\label{eq:problem}
  (-\Delta)^{s} \mathfrak{u} = \mu \quad \text{in } \Omega,
  \qquad
  \mu \in \calM(\Omega).
\end{equation}

Given $s \in (\tfrac{d}{4},1)$, we choose $\theta \in (\tfrac{d}{2} - s, s)$, an interval that is nonempty because $s > \tfrac{d}{4}$. The following inequalities will be fundamental in the analysis below:
\begin{equation}
0 < s -\theta < 2s - \tfrac{d}{2} < 2 - \tfrac{d}{2},
\qquad
\tfrac{d}{2} < s + \theta < 2s < 2.
\end{equation}

We now introduce the bilinear form
\begin{align*}
  \calA &: \mathbb{H}^{s-\theta}(\Omega) \times \mathbb{H}^{s+\theta}(\Omega)
  \to \mathbb{R},
  \\
  \calA(v,w) &\coloneqq \sum_{k=1}^{\infty} \lambda_{k}^{s} v_{k} w_{k},
  \quad
  v = \sum_{k=1}^{\infty} v_{k} \varphi_{k},
  \quad
  w = \sum_{k=1}^{\infty} w_{k} \varphi_{k}.
\end{align*}

\begin{definition}[weak solution]\label{def:weak}
A function $\mathfrak{u} \in \mathbb{H}^{s-\theta}(\Omega)$ is a weak solution of \eqref{eq:problem} if
\begin{equation}\label{eq:weak}
  \calA(\mathfrak{u},v) = \langle \mu , v \rangle
  \qquad \forall v \in \mathbb{H}^{s+\theta}(\Omega) .
\end{equation}
\end{definition}

Since $s + \theta > \frac{d}{2}$, \cite[Theorem 6.7]{MR2944369} shows that $H^{s+\theta-1}(\Omega) \hookrightarrow L^{p}(\Omega)$ for every $p \leq p^{\star} \coloneqq 2d/(d-2(s+\theta - 1))$. Moreover, we have $p^{\star}>d$. Thus, the space characterizations in Section~\ref{sub:spectral} and \cite[Theorem 4.12, Part II]{MR2424078} yield
\begin{equation}
 \mathbb{H}^{s + \theta}(\Omega) = H^{s+\theta}(\Omega) \cap H_0^1(\Omega)
 \hookrightarrow W^{1,p^{\star}}_{0}(\Omega)
 \hookrightarrow C_0(\Omega).
 \label{eq:embeddings}
\end{equation}
Consequently, every $\mu \in \calM(\Omega)$ defines a bounded linear functional on $\mathbb{H}^{s+\theta}(\Omega)$, so the right-hand side of \eqref{eq:weak} is well-defined.

\begin{theorem}[well-posedness]
\label{thm:well_posedness}
Given $\mu \in \mathcal{M}(\Omega)$, problem \eqref{eq:weak} has a unique solution $\mathfrak{u} \in \mathbb{H}^{s-\theta}(\Omega)$ satisfying
\begin{equation}
 \label{eq:stability_bound}
 \| \mathfrak{u} \|_{\mathbb{H}^{s-\theta}(\Omega)} \lesssim \| \mu \|_{\mathbb{H}^{-s-\theta}(\Omega)} \lesssim \| \mu \|_{\mathcal{M}(\Omega)}.
\end{equation}
\end{theorem}
\begin{proof}
The two-dimensional argument of \cite[Theorem 3.2]{OtarolaSalgado2026} extends verbatim to $d=3$: the bilinear form $\calA$ is continuous and satisfies the two inf-sup conditions stated therein, both with constant one. Existence, uniqueness, and the first bound in \eqref{eq:stability_bound} thus follow from the BNB theorem \cite[Theorem 2.2]{MR2648380}. The second bound is a consequence of \eqref{eq:embeddings}.
\end{proof}

\section{The optimal control problem}
\label{sec:optimal_control_problem}

The optimal control problem introduced in Section~\ref{sec:introduction} reads as follows:
\begin{equation}
 \label{eq:min}
 \min \{ J(u, \mathbf{q}) :  (u, \mathbf{q}) \in \mathbb{H}^{s-\theta}(\Omega) \times \mathbf{Q}_{ad} \},
\end{equation}
subject to the weak formulation of the state equation: Find $u \in \mathbb{H}^{s-\theta}(\Omega)$ such that
\begin{equation}
\label{eq:state_equation_weak}
\mathcal{A}(u,v) = \sum_{z \in \mathcal{D}} q_z v(z)
\quad
\forall v \in \mathbb{H}^{s + \theta}(\Omega).
\end{equation}

By Theorem \ref{thm:well_posedness}, the state equation is well-posed. Define the control-to-state map $\mathbf{S}: \mathbb{R}^{\ell} \to \mathbb{H}^{s-\theta}(\Omega)$, which maps $\mathbf{q}$ to the unique solution $u=\mathbf{S}\mathbf{q}$ of \eqref{eq:state_equation_weak}; $\mathbf{S}$ is linear and bounded, since \eqref{eq:stability_bound} yields $\| \mathbf{S} \mathbf{q} \|_{\mathbb{H}^{s-\theta}(\Omega)} \lesssim \| \mathbf{q} \|_{\mathbb{R}^{\ell}}$. Moreover, $s - \theta > 0$ gives $\mathbf{S} \mathbf{q} \in L^{2}(\Omega)$, so that the reduced functional $j : \mathbb{R}^{\ell} \to \mathbb{R}$, defined by $j(\mathbf{q}) \coloneqq J(\mathbf{S}\mathbf{q},\mathbf{q})$, is well defined.

\subsection{Existence and uniqueness of an optimal control}

\begin{theorem}[existence and uniqueness]
\label{thm:existence}
The optimal control problem \eqref{eq:min}--\eqref{eq:state_equation_weak} has a unique solution $(\bar{u},\bar{\mathbf{q}}) \in \mathbb{H}^{s-\theta}(\Omega) \times \mathbf{Q}_{ad}$.
\end{theorem}

\begin{proof}
Using $\mathbf{S}$ and $j$, the control problem reduces to minimizing $j$ over $\mathbf{Q}_{ad}$. The set $\mathbf{Q}_{ad}$ is closed and bounded in $\mathbb{R}^{\ell}$, and hence compact, while $j$ is continuous, because $\mathbf{S}$ is linear and bounded, and strictly convex, because $\alpha > 0$. Existence follows from the Weierstrass theorem, and uniqueness is guaranteed by the strict convexity of $j$.
\end{proof}

\subsection{Optimality conditions}

To derive optimality conditions, we introduce the \emph{adjoint problem}:
\begin{equation}\label{eq:adjoint}
p \in \mathbb{H}^{s}(\Omega):
\qquad
\mathcal{B}(p,w) = (u - u_{\Omega}, w )_{L^2(\Omega)}
\qquad
\forall w \in \mathbb{H}^s(\Omega),
\end{equation}
where the bilinear form $\mathcal{B} : \mathbb{H}^s(\Omega) \times \mathbb{H}^s(\Omega) \rightarrow \mathbb{R}$ is defined by
\[
  \mathcal{B}(v,w) \coloneqq \sum_{k=1}^{\infty} \lambda_{k}^{s} v_{k} w_{k}.
\]
The Lax--Milgram lemma immediately yields the well-posedness of \eqref{eq:adjoint}. Since $u - u_{\Omega} \in L^{2}(\Omega)$, the definitions of $(-\Delta)^s$ and the spaces $\mathbb{H}^{r}(\Omega)$ imply that $p \in \mathbb{H}^{2s}(\Omega)$. Moreover, since $2s > d/2$, the characterization from Section~\ref{sub:spectral} and \cite[Theorem 7.34(c)]{MR2424078} yield $\mathbb{H}^{2s}(\Omega) = H^{2s}(\Omega) \cap H_0^1(\Omega) \hookrightarrow C_0(\Omega)$. Consequently, the point value $p(z)$ is well defined for every $z \in \calD$.

\begin{theorem}[optimality conditions]
\label{thm:optimality}
$\bar{\mathbf{q}} \in \mathbf{Q}_{ad}$ is optimal for
\eqref{eq:min}--\eqref{eq:state_equation_weak} if and only if
\begin{equation}\label{eq:variational_inequality}
  \sum_{z \in \calD} \left( \bar{p}(z) + \alpha \bar{q}_{z} \right)
  \left( q_{z} - \bar{q}_{z} \right) \geq 0,
  \qquad \forall \mathbf{q} = \{ q_z \}_{z \in \mathcal{D}}  \in \mathbf{Q}_{ad},
\end{equation}
where $\bar{p}$ solves \eqref{eq:adjoint} with $u$ replaced by $\bar{u} = \mathbf{S}\bar{\mathbf{q}}$.
\end{theorem}
\begin{proof}
Since $\mathbf{S}$ is linear and bounded, it is Fr\'echet differentiable.  The chain rule then shows that $j$ is differentiable and that, for $\mathbf{q},\mathbf{r} \in \mathbb{R}^{\ell}$,
$
  j'(\mathbf{q}) \mathbf{r}
  = (\mathbf{S}\mathbf{q} - u_{\Omega} , \mathbf{S}\mathbf{r} )_{L^{2}(\Omega)}
  +
  \alpha \sum_{z \in \calD} q_{z} r_{z} .
$
Thus, since $j$ is convex and $\mathbf{Q}_{ad}$ is convex, $\bar{\mathbf{q}} \in \mathbf{Q}_{ad}$ is optimal if and only if, for all $\mathbf{q} = \{ q_z \}_{z \in \mathcal{D}} \in \mathbf{Q}_{ad}$,
\begin{equation}
j'(\bar{\mathbf{q}}) (\mathbf{q} - \bar{\mathbf{q}})
=
( \bar{u} - u_{\Omega} , u - \bar{u}  )_{L^{2}(\Omega)}
+
\alpha \sum_{z \in \calD} \bar{q}_{z} (q_{z} - \bar{q}_z) \geq 0,
\label{eq:basic_variational_inequality}
\end{equation}
where $u = \mathbf{S} \mathbf{q}$ and $\bar{u} = \mathbf{S}\bar{\mathbf{q}}$. It remains to rewrite $( \bar{u} - u_{\Omega} , u - \bar{u}  )_{L^{2}(\Omega)}$ in \eqref{eq:basic_variational_inequality}. By linearity, $u - \bar{u}$ solves \eqref{eq:state_equation_weak} with $\mathbf{q}$ replaced by $\mathbf{q} - \bar{\mathbf{q}}$. Let $\bar{p}$ solve \eqref{eq:adjoint} with $u$ replaced by $\bar{u} = \mathbf{S}\bar{\mathbf{q}}$. Since $s + \theta < 2s$, we have $\bar{p} \in \mathbb{H}^{2s}(\Omega) \hookrightarrow \mathbb{H}^{s+\theta}(\Omega)$. Setting $v = \bar{p}$ in the equation satisfied by $u - \bar{u}$ yields
\begin{equation}
\label{eq:aux_1}
\mathcal{A}(u - \bar u,\bar p) = \sum_{z \in \mathcal{D}} (q_z - \bar{q}_z) \bar{p}(z).
\end{equation}
Let $\{ \mathtt{u}_n \}_{n \in \mathbb{N}} \subset C_0^{\infty}(\Omega)$ be such that $\mathtt{u}_n \to u - \bar{u}$ in $\mathbb{H}^{s-\theta}(\Omega)$ as $n \uparrow \infty$. Using $\mathtt{u}_n$ as a test function in \eqref{eq:adjoint}, we obtain $\mathcal{B}(\bar{p},\mathtt{u}_n) = (\bar{u} - u_{\Omega},\mathtt{u}_n)_{L^2(\Omega)}$ for every $n \in \mathbb{N}$. The convergence $\mathtt{u}_n \to u-\bar{u}$ in $\mathbb{H}^{s-\theta}(\Omega)$, and hence in $L^2(\Omega)$, yields $(\bar{u} - u_{\Omega},\mathtt{u}_n)_{L^2(\Omega)} \rightarrow (\bar{u} - u_{\Omega},u - \bar{u})_{L^2(\Omega)}$ as $n \uparrow \infty$. On the other hand, since $\bar{p} \in \mathbb{H}^{s+\theta}(\Omega)$,
$
 \mathcal{B}(\bar{p},\mathtt{u}_n) = \mathcal{A}(\mathtt{u}_n,\bar{p}) \rightarrow \mathcal{A}(u - \bar{u}, \bar{p})
$
as $n \uparrow \infty$. Combining these limits, we obtain $\mathcal{A}(u - \bar{u}, \bar{p}) = (\bar{u} - u_{\Omega},u - \bar{u})_{L^2(\Omega)}$. We finally invoke \eqref{eq:aux_1} to deduce that
\[
(\bar{u} - u_{\Omega},u - \bar{u})_{L^2(\Omega)}
=
\sum_{z \in \mathcal{D}} (q_z - \bar{q}_z) \bar{p}(z).
\]
Substituting this identity into \eqref{eq:basic_variational_inequality} yields \eqref{eq:variational_inequality} and concludes the proof.
\end{proof}

\section{Discretization}
\label{sec:discretization}

In this section, we develop a finite element discretization of the optimal control problem \eqref{eq:min}--\eqref{eq:state_equation_weak} and derive error bounds.

\subsection{A finite element method for fractional PDEs}
We let $\mathbb{T} = \{ \mathcal{T}_{h} \}_{h>0}$ be a quasiuniform family of conforming simplicial meshes of $\bar{\Omega}$ and denote by $\mathbb{V}(\mathcal{T}_{h})$ the space of continuous piecewise linear functions on $\mathcal{T}_{h}$ that vanish on $\partial \Omega$. We follow \cite{MR5023054}, choose $\mathcal{Y} > 0$ and $K \in \mathbb{N}$, and define the parameters
\begin{equation}
\label{eq:quadrature_parameters}
  \Upsilon_{k} \coloneqq \left( \frac{\eta_{k}}{\mathcal{Y}} \right)^{2},
  \qquad
  \psi_{k} \coloneqq \frac{4 \sin (\pi s)}
                   {\Upsilon_{k}^{s} \mathcal{Y}^{2} \pi J_{1-s}(\eta_{k})^{2}},
  \qquad k = 1,\dots,K,
\end{equation}
where $J_{\nu}$ denotes the Bessel function of the first kind of order $\nu$ and $\eta_{k}$ is the $k$-th positive root of $J_{-s}$. Given $\mathfrak{f} \in L^{2}(\Omega)$, for each $k \in \{ 1, \dots, K \}$, we let $\Psi_{k} \in \mathbb{V}(\mathcal{T}_{h})$ solve
\begin{equation}
\label{eq:reaction_diffusion}
  \int_{\Omega} \nabla \Psi_{k} \cdot \nabla v_{h} \, \mathrm{d}x
  + \Upsilon_{k} \int_{\Omega} \Psi_{k} v_{h} \, \mathrm{d}x
  = \int_{\Omega} \mathfrak{f} v_{h} \, \mathrm{d}x
  \qquad \forall v_{h} \in \mathbb{V}(\mathcal{T}_{h}),
\end{equation}
and define $\Psi_{h,\mathcal{Y}}^{K} \coloneqq \sum_{k=1}^{K} \psi_{k} \Psi_{k} \in \mathbb{V}(\mathcal{T}_{h})$. The following is the main result of \cite{MR5023054}.

\begin{proposition}[$L^{2}$-right-hand side]
\label{prop:L2_data}
Let $\mathfrak{f} \in L^{2}(\Omega)$. If $\mathcal{Y} \eqsim 2s |\log h|$ and $K \eqsim \mathcal{Y}/h$, then
\[
  \| (-\Delta)^{-s} \mathfrak{f} - \Psi_{h,\mathcal{Y}}^{K} \|_{L^{2}(\Omega)}
  \lesssim h^{2s} \| \mathfrak{f} \|_{L^{2}(\Omega)}.
\]
\end{proposition}

When $\mathfrak{f} = \mu \in \mathcal{M}(\Omega)$, we apply the scheme of \cite{MR5023054} to a suitably regularized right-hand side. Given $\varepsilon > 0$, we let $\mu_{\varepsilon} \coloneqq \mu \star \rho_{\varepsilon}$, where $\rho_{\varepsilon}(x) = \varepsilon^{-d} \rho(x/\varepsilon)$ and $\rho \in C_0^\infty(\mathbb{R}^d)$ is a standard mollifier satisfying $\rho \geq 0$, $\operatorname{supp} \rho \subset B_{1}(0)$, and $\int_{\mathbb{R}^d} \rho \, \mathrm{d}x = 1$. For $d \in \{2,3\}$, we have the following crucial estimates:
\begin{align}
 \| \mu_{\varepsilon} \|_{L^2(\Omega)} & \leq \| \rho_{\varepsilon} \|_{L^2(\mathbb{R}^d)} \| \mu \|_{\mathcal{M}(\Omega)}
 =  \varepsilon^{-d/2} \| \rho \|_{L^2(\mathbb{R}^d)} \| \mu \|_{\mathcal{M}(\Omega)},
 \label{eq:regularization_first}
 \\
 \| \mu - \mu_{\varepsilon} \|_{\mathbb{H}^{-s-\theta}(\Omega)} & \lesssim \varepsilon^{\gamma}\| \mu \|_{\mathcal{M}(\Omega)},
 \quad
 \gamma \coloneqq s + \theta - \tfrac{d}{2} > 0.
 \label{eq:regularization_second}
\end{align}
The bound in \eqref{eq:regularization_first} follows from \cite[Proposition 8.49]{MR1681462} and a change of variables, while \eqref{eq:regularization_second} follows upon adapting the arguments in \cite[Proposition 7.4]{OtarolaSalgado2026} to $d \in \{2,3\}$.

Define $\mathfrak{u}_{h,\mathcal{Y}}^{K,\varepsilon} \coloneqq \sum_{k=1}^{K} \psi_{k} \mathfrak{u}_{k}^{\varepsilon} \in \mathbb{V}(\mathcal{T}_{h})$, where $\mathfrak{u}_k^{\varepsilon}$ solves \eqref{eq:reaction_diffusion} with $\mathfrak{f}$ replaced by $\mu_{\varepsilon}$.

\begin{theorem}[measure-valued right-hand side]
\label{thm:measure_data}
Let $\mu \in \calM(\Omega)$, and let $\mathfrak{u}$ solve \eqref{eq:weak}. If $\mathcal{Y} \eqsim 2s|\log h|$, $K \eqsim \mathcal{Y}/h$ and
$\varepsilon \eqsim h$, then, for $h \leq 1$, we have
\begin{equation}
\label{eq:rate_measure_data}
  \| \mathfrak{u} - \mathfrak{u}_{h,\mathcal{Y}}^{K,\varepsilon} \|_{L^{2}(\Omega)}
  \lesssim h^{s+\theta-\frac{d}{2}} \| \mu \|_{\calM(\Omega)} .
\end{equation}
\end{theorem}
\begin{proof}
Let $\mathfrak{u}_{\varepsilon}$ solve $(-\Delta)^{s}\mathfrak{u}_{\varepsilon}
= \mu_{\varepsilon}$ in $\Omega$. Proposition~\ref{prop:L2_data} gives $\| \mathfrak{u}_{\varepsilon} - \mathfrak{u}_{h,\mathcal{Y}}^{K,\varepsilon}
\|_{L^{2}(\Omega)} \lesssim h^{2s} \| \mu_{\varepsilon} \|_{L^2(\Omega)} \lesssim h^{2s} \varepsilon^{-d/2} \| \mu
\|_{\calM(\Omega)}$, where we also used \eqref{eq:regularization_first}. Since $\mathfrak{u} - \mathfrak{u}_{\varepsilon}$ solves problem \eqref{eq:weak} with
$\mu$ replaced by $\mu - \mu_{\varepsilon}$, the continuous embedding $\mathbb{H}^{s-\theta}(\Omega) \hookrightarrow L^2(\Omega)$, together with \eqref{eq:stability_bound} and \eqref{eq:regularization_second}, yields $\| \mathfrak{u} - \mathfrak{u}_{\varepsilon}
\|_{L^{2}(\Omega)} \lesssim
\| \mathfrak{u} - \mathfrak{u}_{\varepsilon}
\|_{\mathbb{H}^{s-\theta}(\Omega)}
\lesssim
\varepsilon^{\gamma} \| \mu \|_{\calM(\Omega)}$. Setting $\varepsilon \eqsim h$ and using $\theta < s$ and $h \leq 1$, we arrive at $h^{2s-d/2} \leq
h^{\gamma}$. The triangle inequality then yields \eqref{eq:rate_measure_data}.
\end{proof}

\subsection{A finite element method for the optimal control problem}
\label{sub:discrete_control}

Recall that, for the optimal control problem, the state $u$ solves \eqref{eq:problem} with right-hand side given by
\[
  \mu = \sum_{z \in \calD} q_{z} \delta_{z}.
\]
For each $z \in \calD$ and $\varepsilon>0$, we define $\delta_{z,\varepsilon} \coloneqq \delta_{z} \star \rho_{\varepsilon}$. We then set
\[
  \mu_{\varepsilon} \coloneqq \sum_{z \in \calD} q_{z} \delta_{z,\varepsilon}.
\]
Next, following Theorem~\ref{thm:measure_data}, we choose the discretization and regularization parameters so that $\mathcal{Y} \eqsim 2s|\log h|$, $K \eqsim \mathcal{Y}/h$, and $\varepsilon \eqsim h$. We define the discrete control-to-state map
\begin{equation}
  \mathbf{S}_{h} : \mathbb{R}^{\ell} \to \mathbb{V}(\mathcal{T}_{h}),
  \qquad
  \mathbf{S}_{h}\mathbf{q} \coloneqq u_{h,\mathcal{Y}}^{K,\varepsilon}.
  \label{eq:discrete_control_to_state_map}
\end{equation}
Since $\mu_{\varepsilon}$ depends linearly on $\mathbf{q}$, each problem \eqref{eq:reaction_diffusion} is linear with respect to its right-hand side, and the weights $\psi_{k}$ in \eqref{eq:quadrature_parameters} are independent of $\mathbf{q}$, the map $\mathbf{S}_{h}$ is linear.  Moreover, since $\| \mu \|_{\calM(\Omega)} \lesssim \| \mathbf{q} \|_{\mathbb{R}^{\ell}}$, Theorem~\ref{thm:measure_data} yields
\begin{equation}
\label{eq:discrete_state_error}
  \| ( \mathbf{S} - \mathbf{S}_{h} ) \mathbf{q} \|_{L^{2}(\Omega)}
  \lesssim h^{s + \theta - \frac{d}{2}} \| \mathbf{q} \|_{\mathbb{R}^{\ell}},
  \qquad
  h \leq 1.
\end{equation}
In addition, the family $\{\mathbf{S}_{h}\}_{0<h\leq1}$ is uniformly bounded in $h$ as a family of operators from $\mathbb{R}^{\ell}$ into $L^{2}(\Omega)$. Indeed,
\[
  \| \mathbf{S}_{h} \mathbf{q} \|_{L^{2}(\Omega)} \leq
\| ( \mathbf{S} - \mathbf{S}_{h} ) \mathbf{q} \|_{L^{2}(\Omega)}
+ \| \mathbf{S}\mathbf{q} \|_{L^{2}(\Omega)} \lesssim
\| \mathbf{q} \|_{\mathbb{R}^{\ell}}.
\]

We thus propose the following discretization of the control problem \eqref{eq:min}--\eqref{eq:state_equation_weak}: Find $\bar{\mathbf q}_h\in\mathbf Q_{ad}$ such that
\begin{equation}
\label{eq:discrete_min}
  j_h(\bar{\mathbf q}_h) = \min \left\{ j_{h}(\mathbf{q}) : \mathbf{q} \in \mathbf{Q}_{ad} \right\},
  \qquad
  j_{h}(\mathbf{q}) \coloneqq \tfrac{1}{2}
  \| \mathbf{S}_{h}\mathbf{q} - u_{\Omega} \|_{L^{2}(\Omega)}^{2}
  + \tfrac{\alpha}{2} \| \mathbf{q} \|_{\mathbb{R}^{\ell}}^{2},
\end{equation}
where $\mathbf{S}_{h}$ is the discrete control-to-state map defined in \eqref{eq:discrete_control_to_state_map}.
Since $j_h$ is continuous and strictly convex and $\mathbf{Q}_{ad}$ is compact and convex, problem \eqref{eq:discrete_min} admits a unique optimal solution $\bar{\mathbf{q}}_{h} \in \mathbf{Q}_{ad}$. Moreover, $\bar{\mathbf{q}}_{h} \in \mathbf{Q}_{ad}$ is optimal for \eqref{eq:discrete_min} if and only if
\begin{equation}
\label{eq:discrete_variational_inequality}
  j_{h}'(\bar{\mathbf{q}}_{h}) (\mathbf{q} - \bar{\mathbf{q}}_{h}) \geq 0
  \qquad \forall \mathbf{q} \in \mathbf{Q}_{ad} .
\end{equation}

We now present the main result of this section.

\begin{theorem}[error bounds]
\label{thm:convergence}
Let $\bar{\mathbf{q}}$ be the optimal solution of the optimal control problem \eqref{eq:min}--\eqref{eq:state_equation_weak}, and let $\bar{\mathbf{q}}_{h}$ be the optimal solution of the discrete optimal control problem \eqref{eq:discrete_min}. Then, for $h \leq 1$, we have
\begin{equation}
\label{eq:control_convergence}
  \| \bar{\mathbf{q}} - \bar{\mathbf{q}}_{h} \|_{\mathbb{R}^{\ell}}
  \lesssim  h^{s+\theta-\frac{d}{2}},
  \qquad
  \| \bar{u} - \mathbf{S}_{h} \bar{\mathbf{q}}_{h} \|_{L^{2}(\Omega)}
  \lesssim h^{s+\theta-\frac{d}{2}},
\end{equation}
where the hidden constants are independent of $h$ and may depend, in particular, on $\ell$, $\alpha$, $\mathbf a$, $\mathbf b$, and $\|u_\Omega\|_{L^2(\Omega)}$.
\end{theorem}
\begin{proof}
Basic computations show the following identity: $(j'(\bar{\mathbf{q}}) - j'(\bar{\mathbf{q}}_h)) (\bar{\mathbf{q}} - \bar{\mathbf{q}}_h) = \alpha \| \bar{\mathbf{q}} - \bar{\mathbf{q}}_{h} \|_{\mathbb{R}^{\ell}}^{2} + \| \mathbf{S}(\bar{\mathbf{q}} - \bar{\mathbf{q}}_h)\|^2_{L^2(\Omega)}$. As a result, we obtain the bound
\[
\alpha \| \bar{\mathbf{q}} - \bar{\mathbf{q}}_{h} \|_{\mathbb{R}^{\ell}}^{2}
\leq
(j'(\bar{\mathbf{q}}) - j'(\bar{\mathbf{q}}_h)) (\bar{\mathbf{q}} - \bar{\mathbf{q}}_h).
\]
On the other hand, if we set $\mathbf{q} = \bar{\mathbf{q}}_{h}$ in
\eqref{eq:basic_variational_inequality} and $\mathbf{q} = \bar{\mathbf{q}}$ in \eqref{eq:discrete_variational_inequality}, we obtain  $j'(\bar{\mathbf{q}}) (\bar{\mathbf{q}} - \bar{\mathbf{q}}_{h}) \leq 0$ and
$j_{h}'(\bar{\mathbf{q}}_{h}) (\bar{\mathbf{q}} - \bar{\mathbf{q}}_{h}) \geq
0$, respectively. Combining these inequalities with the previous bound, we obtain
\begin{equation}
\label{eq:instrumental}
  \alpha \| \bar{\mathbf{q}} - \bar{\mathbf{q}}_{h}
  \|_{\mathbb{R}^{\ell}}^{2}
  \leq
  \bigl( j_{h}'(\bar{\mathbf{q}}_{h}) - j'(\bar{\mathbf{q}}_{h}) \bigr)
  (\bar{\mathbf{q}} - \bar{\mathbf{q}}_{h}).
\end{equation}
We now add and subtract to the right-hand side of the previous bound the term $( \mathbf{S}\bar{\mathbf{q}}_h - u_{\Omega} , \mathbf{S}_{h}(\bar{\mathbf{q}} - \bar{\mathbf{q}}_h))_{L^{2}(\Omega)}$ to obtain
\[
 \alpha \| \bar{\mathbf{q}} - \bar{\mathbf{q}}_{h}
  \|_{\mathbb{R}^{\ell}}^{2}
  \leq
  ( (\mathbf{S}_h  - \mathbf{S}) \bar{\mathbf{q}}_h, \mathbf{S}_{h}(\bar{\mathbf{q}} - \bar{\mathbf{q}}_h))_{L^{2}(\Omega)}
  +
  ( \mathbf{S} \bar{\mathbf{q}}_h - u_{\Omega}, (\mathbf{S}_{h}-\mathbf{S})(\bar{\mathbf{q}} - \bar{\mathbf{q}}_h))_{L^{2}(\Omega)}.
\]
Since $\{\mathbf{S}_{h}\}_{0<h\leq1}$ is uniformly bounded in $h$, and the quantities $\| \bar{\mathbf{q}}_{h} \|_{\mathbb{R}^{\ell}}$ and
$\| \mathbf{S}\bar{\mathbf{q}}_{h} - u_{\Omega} \|_{L^{2}(\Omega)}$ are bounded uniformly in $h$ as well because $\bar{\mathbf{q}}_{h} \in \mathbf{Q}_{ad}$, we obtain
\[
\alpha \| \bar{\mathbf{q}} - \bar{\mathbf{q}}_{h}
  \|^2_{\mathbb{R}^{\ell}}
  \lesssim
  \| (\mathbf{S}_h  - \mathbf{S}) \bar{\mathbf{q}}_h \|_{L^2(\Omega)}
  \|\bar{\mathbf{q}} - \bar{\mathbf{q}}_h \|_{\mathbb{R}^{\ell}}
  +
  \| (\mathbf{S}_{h}-\mathbf{S})(\bar{\mathbf{q}} - \bar{\mathbf{q}}_h)\|_{L^{2}(\Omega)}
  \lesssim h^{\gamma} \| \bar{\mathbf{q}} - \bar{\mathbf{q}}_{h}
  \|_{\mathbb{R}^{\ell}},
\]
where we have also used \eqref{eq:discrete_state_error}. This yields the first estimate in
\eqref{eq:control_convergence}. The second one follows from
$
  \| \bar{u} - \mathbf{S}_{h} \bar{\mathbf{q}}_{h} \|_{L^{2}(\Omega)}
  \leq
  \| \mathbf{S} ( \bar{\mathbf{q}} - \bar{\mathbf{q}}_{h} )
  \|_{L^{2}(\Omega)}
  +
  \| ( \mathbf{S} - \mathbf{S}_{h} ) \bar{\mathbf{q}}_{h}
  \|_{L^{2}(\Omega)}
  \lesssim
  \| \bar{\mathbf{q}} - \bar{\mathbf{q}}_{h} \|_{\mathbb{R}^{\ell}}
  + h^{\gamma} \lesssim h^{\gamma},
$
where we have used the first estimate in \eqref{eq:control_convergence}.
\end{proof}

\begin{remark}[on the rate]
\label{rem:rate}
An appropriate choice of $\theta$ in \eqref{eq:control_convergence} yields the rate $\mathcal{O}(h^{2s-\frac{d}{2}-\delta})$, with $\delta > 0$ arbitrarily small, which is not optimal. For $s \uparrow 1$ and $d=2$ it approaches $\mathcal{O}(h^{1-\delta})$, whereas the corresponding local problem admits $\mathcal{O}(h^{2}|\log h|^{3})$ \cite[Theorem 5.8]{MR4838450}. Our argument only exploits the $L^{2}(\Omega)$-error of the state approximation, while the arguments therein rely on local $L^{\infty}$ estimates. To the best of our knowledge, local $L^{\infty}$-error bounds are not available for FE approximations of $(-\Delta)^s\phi=g$, and obtaining them appears to be an open problem of independent interest.
\end{remark}

\section*{Funding}

EO: USM, CHILE through USM project 2026 PI LIR 26 06.

AJS: National Science Foundation, USA grant DMS-2409918.

\section*{Data availability}
No data was used for the research described in the article.

\bibliographystyle{plain}
\bibliography{biblio}

@article {MR2086168,
    AUTHOR = {Berm\'udez, A. and Gamallo, P. and Rodr\'iguez,
              R.},
     TITLE = {Finite element methods in local active control of sound},
   JOURNAL = {SIAM J. Control Optim.},
  FJOURNAL = {SIAM Journal on Control and Optimization},
    VOLUME = {43},
      YEAR = {2004},
    NUMBER = {2},
     PAGES = {437--465},
      ISSN = {0363-0129,1095-7138},
   MRCLASS = {49J20 (49K20 65N30 76M10 76N25 76Q05)},
  MRNUMBER = {2086168},
MRREVIEWER = {Ruxandra\ Stavre},
       DOI = {10.1137/S0363012903431785},
       URL = {https://doi.org/10.1137/S0363012903431785},
}

@article {MR4838450,
    AUTHOR = {Ot\'arola, E.},
     TITLE = {Semilinear optimal control with {D}irac measures},
   JOURNAL = {IMA J. Numer. Anal.},
  FJOURNAL = {IMA Journal of Numerical Analysis},
    VOLUME = {44},
      YEAR = {2024},
    NUMBER = {6},
     PAGES = {3573--3594},
      ISSN = {0272-4979,1464-3642},
   MRCLASS = {49J20 (35J25 35J61 35R06 49K20 49M41)},
  MRNUMBER = {4838450},
MRREVIEWER = {Lino\ J.\ \'Alvarez-V\'azquez},
       DOI = {10.1093/imanum/drad091},
       URL = {https://doi.org/10.1093/imanum/drad091},
}

@book {MR1681462,
    AUTHOR = {Folland, G.~B.},
     TITLE = {Real analysis},
    SERIES = {Pure and Applied Mathematics (New York)},
   EDITION = {Second},
      NOTE = {Modern techniques and their applications,
              A Wiley-Interscience Publication},
 PUBLISHER = {John Wiley \& Sons, Inc., New York},
      YEAR = {1999},
     PAGES = {xvi+386},
      ISBN = {0-471-31716-0},
   MRCLASS = {00A05 (26-01 28-01 46-01)},
  MRNUMBER = {1681462},
}

@article {MR5023054,
    AUTHOR = {Salgado, A.~J. and Sawyer, S.~E.},
     TITLE = {A semianalytic diagonalization finite element method for the
              spectral fractional {L}aplacian},
   JOURNAL = {SIAM J. Sci. Comput.},
  FJOURNAL = {SIAM Journal on Scientific Computing},
    VOLUME = {48},
      YEAR = {2026},
    NUMBER = {1},
     PAGES = {A236--A260},
      ISSN = {1064-8275,1095-7197},
   MRCLASS = {65N30 (26A33 33C10 35R11 41A55 65N12)},
  MRNUMBER = {5023054},
       DOI = {10.1137/24M1699309},
       URL = {https://doi.org/10.1137/24M1699309},
}

@misc{OtarolaSalgado2026,
  author       = {Ot\'arola, E. and Salgado, A.~J.},
  title        = {The spectral fractional {L}aplacian with measure valued
                  right hand sides: analysis and approximation},
  year         = {2026},
  eprint       = {2602.11423},
  archivePrefix = {arXiv},
  primaryClass = {math.NA},
  howpublished = {arXiv:2602.11423}
}

@incollection {MR2648380,
    AUTHOR = {Nochetto, R.~H. and Siebert, K.~G. and Veeser,
              A.},
     TITLE = {Theory of adaptive finite element methods: an introduction},
 BOOKTITLE = {Multiscale, nonlinear and adaptive approximation},
 PUBLISHER = {Springer, Berlin},
      YEAR = {2009},
      ISBN = {978-3-642-03412-1},
   MRCLASS = {65N30 (65N15 65N50)},
  MRNUMBER = {2648380},
MRREVIEWER = {Nicolae\ Pop},
       DOI = {10.1007/978-3-642-03413-8\_12},
       URL = {https://doi.org/10.1007/978-3-642-03413-8_12},
}

@book {MR2424078,
    AUTHOR = {Adams, R.~A. and Fournier, J.~F.},
     TITLE = {Sobolev spaces},
    SERIES = {Pure and Applied Mathematics (Amsterdam)},
    VOLUME = {140},
   EDITION = {Second},
 PUBLISHER = {Elsevier/Academic Press, Amsterdam},
      YEAR = {2003},
     PAGES = {xiv+305},
      ISBN = {0-12-044143-8},
   MRCLASS = {46E35 (46-01 46-02 46B70 46Exx)},
  MRNUMBER = {2424078},
}

@article {MR2944369,
    AUTHOR = {Di Nezza, E. and Palatucci, G. and Valdinoci,
              E.},
     TITLE = {Hitchhiker's guide to the fractional {S}obolev spaces},
   JOURNAL = {Bull. Sci. Math.},
  FJOURNAL = {Bulletin des Sciences Math\'ematiques},
    VOLUME = {136},
      YEAR = {2012},
    NUMBER = {5},
     PAGES = {521--573},
      ISSN = {0007-4497,1952-4773},
   MRCLASS = {46E35 (35A23 35S05 35S30)},
  MRNUMBER = {2944369},
MRREVIEWER = {Lanzhe\ Liu},
       DOI = {10.1016/j.bulsci.2011.12.004},
       URL = {https://doi.org/10.1016/j.bulsci.2011.12.004},
}

@article {MR3356020,
    AUTHOR = {Bonito, A. and Pasciak, J.~E.},
     TITLE = {Numerical approximation of fractional powers of elliptic
              operators},
   JOURNAL = {Math. Comp.},
  FJOURNAL = {Mathematics of Computation},
    VOLUME = {84},
      YEAR = {2015},
    NUMBER = {295},
     PAGES = {2083--2110},
      ISSN = {0025-5718,1088-6842},
   MRCLASS = {65N30 (65R20)},
  MRNUMBER = {3356020},
MRREVIEWER = {Igor\ Bock},
       DOI = {10.1090/S0025-5718-2015-02937-8},
       URL = {https://doi.org/10.1090/S0025-5718-2015-02937-8},
}

@article {MR3343061,
    AUTHOR = {Chandler-Wilde, S. N. and Hewett, D. P. and Moiola, A.},
     TITLE = {Interpolation of {H}ilbert and {S}obolev spaces: quantitative
              estimates and counterexamples},
   JOURNAL = {Mathematika},
  FJOURNAL = {Mathematika. A Journal of Pure and Applied Mathematics},
    VOLUME = {61},
      YEAR = {2015},
    NUMBER = {2},
     PAGES = {414--443},
      ISSN = {0025-5793,2041-7942},
   MRCLASS = {46B70 (46E35)},
  MRNUMBER = {3343061},
MRREVIEWER = {Oscar\ Dom\'inguez},
       DOI = {10.1112/S0025579314000278},
       URL = {https://doi.org/10.1112/S0025579314000278},
}

@book {MR1742312,
    AUTHOR = {McLean, W.},
     TITLE = {Strongly elliptic systems and boundary integral equations},
 PUBLISHER = {Cambridge University Press, Cambridge},
      YEAR = {2000},
     PAGES = {xiv+357},
      ISBN = {0-521-66332-6; 0-521-66375-X},
   MRCLASS = {35J45 (47F05 47G10 47N20 65N38)},
  MRNUMBER = {1742312},
MRREVIEWER = {Dorina\ I.\ Mitrea},
}

@book {MR2328004,
    AUTHOR = {Tartar, Luc},
     TITLE = {An introduction to {S}obolev spaces and interpolation spaces},
    SERIES = {Lecture Notes of the Unione Matematica Italiana},
    VOLUME = {3},
 PUBLISHER = {Springer, Berlin; UMI, Bologna},
      YEAR = {2007},
     PAGES = {xxvi+218},
      ISBN = {978-3-540-71482-8; 3-540-71482-0},
   MRCLASS = {46E35 (41A05 46B70 46M35)},
  MRNUMBER = {2328004},
MRREVIEWER = {Joan\ L.\ Cerd\`a},
}

@article {MR2754080,
    AUTHOR = {Stinga, P.~R. and Torrea, J.~L.},
     TITLE = {Extension problem and {H}arnack's inequality for some
              fractional operators},
   JOURNAL = {Comm. Partial Differential Equations},
  FJOURNAL = {Communications in Partial Differential Equations},
    VOLUME = {35},
      YEAR = {2010},
    NUMBER = {11},
     PAGES = {2092--2122},
      ISSN = {0360-5302,1532-4133},
   MRCLASS = {35R11 (35B45 35B50 35B51 35B65)},
  MRNUMBER = {2754080},
MRREVIEWER = {Nasser-eddine\ Tatar},
       DOI = {10.1080/03605301003735680},
       URL = {https://doi.org/10.1080/03605301003735680},
}

@article {MR2646117,
    AUTHOR = {Cabr\'e, X. and Tan, J.},
     TITLE = {Positive solutions of nonlinear problems involving the square
              root of the {L}aplacian},
   JOURNAL = {Adv. Math.},
  FJOURNAL = {Advances in Mathematics},
    VOLUME = {224},
      YEAR = {2010},
    NUMBER = {5},
     PAGES = {2052--2093},
      ISSN = {0001-8708,1090-2082},
   MRCLASS = {35J10 (35S05)},
  MRNUMBER = {2646117},
MRREVIEWER = {B.\ Kellogg},
       DOI = {10.1016/j.aim.2010.01.025},
       URL = {https://doi.org/10.1016/j.aim.2010.01.025},
}

@article {MR2825595,
    AUTHOR = {Capella, A. and D\'avila, J. and Dupaigne, L. and
              Sire, Y.},
     TITLE = {Regularity of radial extremal solutions for some non-local
              semilinear equations},
   JOURNAL = {Comm. Partial Differential Equations},
  FJOURNAL = {Communications in Partial Differential Equations},
    VOLUME = {36},
      YEAR = {2011},
    NUMBER = {8},
     PAGES = {1353--1384},
      ISSN = {0360-5302,1532-4133},
   MRCLASS = {35R11 (35B65 35D30 35J25 35J91 35S10)},
  MRNUMBER = {2825595},
MRREVIEWER = {William\ Margulies},
       DOI = {10.1080/03605302.2011.562954},
       URL = {https://doi.org/10.1080/03605302.2011.562954},
}

\end{document}